\documentclass{elsarticle}
\pdfoutput=1

\usepackage{amsmath}
\usepackage{amssymb}
\usepackage{amsthm}
\usepackage[utf8]{inputenc} 
\allowdisplaybreaks[4]

\usepackage{graphicx} 
\graphicspath{{figures/}}
\usepackage{color}

\newcommand{\bbE}{\mathbb{E}}
\newcommand{\bbP}{\mathbb{P}}
\newcommand{\bbR}{\mathbb{R}}
\newcommand{\bbN}{\mathbb{N}}

\newcommand{\calF}{\mathcal{F}}

\DeclareMathOperator{\PA}{\text{PA}}
\DeclareMathOperator{\Deg}{\text{Deg}}
\DeclareMathOperator{\argmax}{arg\,max}
\DeclareMathOperator{\Var}{Var}

\newcommand{\lrbkt}[1]{\{#1\}}

\catcode`\@=11 

\newtheorem{theorem}{Theorem}
\newtheorem{lemma}{Lemma}
\newtheorem{proposition}{Proposition}

                                           \def\NN{\mathbb{N}}

 				\def\RR{\mathbb{R}}

\def\ra{\rightarrow}

\def\dotfil{\leaders\hbox to 1em{\hss.\hss}\hfill}

\def\hexnumber#1{\ifcase#1 0\or 1\or 2\or 3\or 4\or 5\or 6\or 7\or 8\or 9\or A\or B\or C\or D\or E\or F\fi}

\def\BL{{\rm B\kern -0.5pt L}}
\def\UC{{\rm U\kern-0.5pt C}}

\def\argmax{\mathop{\rm argmax}}

\def\tdot#1{\kern1.2pt\dot{\vphantom{#1}}\kern-1.2pt
        \dot#1\kern0.8pt\dot{\vphantom{#1}}\kern-0.8pt}

\let\Covar\Cov

\mathchardef\given="626A
\mathcode`:="603A

\def\generalweak#1{\ {\mathchoice{\buildrel #1\over \rightsquigarrow}%
{\raise-2pt\hbox{$\buildrel #1\over \rightsquigarrow$}}{}{}}\ }

\def\generalprob#1{\ {\mathchoice{\raise-1.5pt\hbox{$\buildrel#1\over\ra$}}
{\raise-2pt\hbox{$\buildrel#1\over\ra$}}{}{}}\ }

\def\Pgg{\ {\mathchoice{\raise-1.5pt\hbox{$\buildrel {\small P}\over\gg$}}
{\raise-2pt\hbox{$\buildrel {\small P}\over\gg$}}{}{}}\ }

\def\boxit#1{\vbox{\hrule\hbox{\vrule\kern1pt\vbox{\kern1pt #1\kern1pt}\kern1pt\vrule}\hrule}}
\def\eet#1{}

\catcode`\@=12

\usepackage{lineno,hyperref}
\modulolinenumbers[5]

\journal{Stochastic Processes and their Applications}

\begin{document} 
\begin{frontmatter}
\title{On the Asymptotic Normality of Estimating the Affine Preferential Attachment Network Models with Random Initial Degrees}

\author[fudan,scms]{Fengnan~Gao\corref{corfn}\fnref{fngao}}
\ead{fngao@fudan.edu.cn}

\author[leiden]{Aad~van der Vaart\fnref{fnaad}}
\ead{avdvaart@math.leidenuniv.nl}

\fntext[fngao]{Research supported by the Netherlands Organization for Scientific Research and by NNSF of China Grant 1169010013.}
\fntext[fnaad]{The research leading to these results has received funding from the European Research Council under ERC Grant Agreement 320637.}
\cortext[corfn]{Corresponding author}

\address[fudan]{School of Data Science, Fudan University, Handan Road 220, Shanghai 200433, China}
\address[scms]{Shanghai Center for Mathematical Sciences, Handan Road 220, Shanghai 200433, China}
\address[leiden]{Mathematical Institute, Leiden University, P.O. Box 9512, 2300 RA Leiden, 
The Netherlands}
\begin{abstract}
We consider the estimation of the affine parameter and power-law exponent in the preferential attachment model 
with random initial degrees.  
We derive the likelihood, and show that the maximum likelihood estimator (MLE) is asymptotically normal and efficient. 
We also propose a quasi-maximum-likelihood estimator (QMLE) to overcome the MLE's dependence on the history of the initial degrees.  
To demonstrate the power of our idea, we present numerical simulations.  
\end{abstract}
\end{frontmatter}

\section{Introduction and Notations}
\label{sec:intro}
In the past decade random graphs have become well established for modelling complex
networks.  The \emph{preferential attachment} (PA) model, introduced in
\cite{Barabasi99emergenceScaling}, is popular in studies of social networks, the internet,
collaboration networks, and so on.  The PA model is a \emph{dynamic} model, in that it describes
the evolution of the network through the sequential addition of new nodes, and can explain 
the so-called \textit{scale-free} phenomenon. This is the observation 
that in various real world networks the proportion $p_k$ of nodes of degree $k$ follows
a power law 
\[
    p_k \propto k^{-\tau}, 
\]
for some \emph{power-law exponent} $\tau$.  
For example, Table 3.1 in \cite{newman2003structure} gives comprehensive lists of basic statistics for a number of well-known networks, where the power-law exponent is estimated to be  $2.4$ for protein interactions and $2.5$ for the Internet.  
Another source \cite{faloutsos1999power} estimates the power-law exponent of the Internet to be between $2.15$ and $2.20$. 

The PA model is built upon the simple paradigm  ``the rich get richer'' and offers a possible
scenario in which the Matthew effect (``to all those who have, more will be given'')
takes place (\cite{Perc20140378}).  If the network is modelled as 
a graph, with the vertices representing individuals and the degree of a vertex 
(the number of edges) representing wealth, then this means that a new vertex is more likely to connect to already
well-connected vertices: vertices with higher degrees (``rich'') inspire more
incoming connections (``get richer'').  

In the simplest type of PA model this is implemented as follows.  
We are given a non-decreasing \textit{preferential
  attachment function} $f: \bbN^{+} \rightarrow \mathbb{R}_+$.  The network is initialised at $t=1$ as a graph consisting of two vertices 
with one edge between them.  Then the recursive attachment scheme begins  ($t=2,3,\ldots$).  At time $t$ a new vertex is added to the graph
and is connected to exactly one of the $t$ existing vertices, say $i$, with probability
proportional to $f(d_i)$, where $d_i$ is the degree of vertex $i$ in the graph at time $t-1$.  

The proportionality, which entails normalizing
by the sum of all transformed degrees $f(d)$, makes that an \emph{affine} function $f$ can be parametrized  
without loss of generality by a single parameter $\delta$, in the form
$f(k) = k + \delta$, where minimally $\delta >-1$.  This special case has been well studied. In particular, it has been established (see
e.g.\ \cite{mori2002random,hofstadcomplexnet}) that the \emph{empirical degree distribution} $(p_k(t))_{k=1}^\infty$,
where $p_k(t)$ is the proportion of vertices of degree $k$ in the tree at time $t$, converges to a limiting degree distribution
$(p_k)_{k=1}^\infty $ as $t \rightarrow \infty$, which follows a power law
$$p_k \propto k^{-(3+\delta)}.$$ 
The Barabási–Albert model is the special case that $\delta=0$ and has $p_k = 4/(k(k+1)(k+2))$.


If the limiting degree distribution $(p_k)_{k=1}^\infty$ follows a power law, say
\( p_k = c_k k^{-(3+\delta)}\)
with $c_k$ slowly varying in $k$, then \( \log p_k = \log c_k - (3+\delta) \log k \), where
$\log c_k$ behaves like a constant when $k$ is sufficiently large. This suggests that we might estimate $\delta$ by performing a linear
regression of the logarithm empirical degree $\log p_k(t)$ with respect to $\log k$, with $k$
sufficiently large. Then $\hat \delta := \hat \tau - 3$,  for $\hat{\tau}$ the negative of the slope of the fitted line, should
estimate $\delta$. This has been the method of choice in several applications of PA models; see for example the famous paper 
\cite{Barabasi99emergenceScaling}.  However, this estimator does not work well. One problem is that the equality
\( \log p_k(n) = \log c_k - \tau \log k\)
comes from an asymptotic approximation based on Stirling's formula, and it is unclear when the asymptotics start to kick in.
It is also unclear how to determine the quality of this ad-hoc estimate,  by a standard error or confidence interval, or to perform inference.

In the present paper we remedy this by considering the maximum likelihood estimator of $\delta$.
We show that this can be easily computed as the solution of an equation, prove its asymptotic normality and derive its
standard error. Furthermore, in a simulation study we show that it is quite accurate.

We do this in the following more general context of the \emph{PA model with random initial degrees} as considered in \cite{deijfen2009preferential}. 
In this model the graph also grows by sequentially adding vertices.  However, instead of connecting
a new vertex to a single existing vertex, the vertex at time $t$ comes with $m_t\ge 1$ edges
to connect it to the existing network.  The connections are still made according to the basic preferential attachment rule. 
The special case of this model with each $m_t$ equal to a fixed number $m\in\bbN$ was well studied 
(see \cite{hofstadcomplexnet} for a discussion and  more references).  
For $m=1$ this model reduces to the basic model considered in the preceding paragraphs.
However, for real life applications it is somewhat unnatural and inflexible to assume that every vertex comes with the same number of edges,
let alone a single edge.
In this paper, we allow $m_1,m_2,\ldots$ to be an \textsc{iid} sequence of random variables, only restricted to have
two finite moments. From the point of view of social network modeling the resulting model can be interpreted as follows.
The initial degrees $m_t$ stress the difference of nodes at their times of inclusion (``birth''), while the preferential attachment 
enforces 
the discrimination against the less well-off individuals (or equivalently the positive discrimination against the well-off individuals).  
Thus the PA network with random initial degrees combines the effects of ``rich-by-birth'' and ``rich-get-richer''.   

Generally this model also gives networks with degree distributions of asymptotic power law type,
but the exponent depends on both the attachment function and the initial degree distribution.
Suppose that the preferential attachment function is $f(k) = k + \delta$ as before, and
denote the mean of the  initial degree distribution by $\mu=\bbE m_t$. Then \cite{deijfen2009preferential} 
has shown that the limiting degree distribution $(p_k)$ for the model follows asymptotically (as $k\ra\infty$)
a power law with exponent $\min(3+\delta/\mu, \tau_m)$ if the random initial degree distribution
follows a power law with exponent $\tau_m$ and is equal to $3+\delta/\mu$ if the latter distribution
decays faster than a power law (which is equivalent to setting $\tau_m = \infty$).  

The estimation of a preferential attachment function $f$ is of both theoretical and practical interest.  
Despite the omnipresence of PA networks in modeling real-world networks, 
the literature on statistical estimation of this model is sparse.
Sound understanding of the statistical properties of PA models will help in evaluating the validity of the PA network model in real-world applications.

The paper \cite{empiestimator} proposed an empirical estimator for each $f(k)$ given
a general sub-linear preferential attachment function $f$. Although
this estimator was proved to be consistent as $t\ra\infty$, we may expect better  estimators
if we restrict attention to the domain of affine functions, whose estimation is equivalent to estimating
the single parameter $\delta$. A restriction to affine functions is not unnatural, 
as exactly the affine preferential attachment models correspond to power-law behaviour. Therefore affine $f$ matter
most to the practitioners seeking a natural explanation of power laws.  The paper
\cite{2015arXiv150407328R} solves some statistical problems in the affine model, but focuses on the estimation of the individual
degree distribution, and does not directly consider estimating the key affine
parameter $\delta$.  
The present paper provides a ready-to-use solution for estimating the affine parameter, and sound  theoretical insight into the statistical modeling of PA networks.

The estimation problem of the affine parameter $\delta$ is not only interesting in itself, but also important because of its direct connection to estimating the power-law exponent. Although power laws are ubiquitous, as pointed out above, estimating their components can be difficult (\cite{doi:10.1137/070710111}).  
One must deal with two types of asymptotics here: the limiting degree distribution when the number of vertices goes to infinity and the limiting power-law behavior as the degree goes to infinity in the limiting degree distribution.
It is hard to determine when these asymptotics kick in and hence are usable for an estimation procedure.
For instance, \cite{doi:10.1137/070710111} proposed an estimator for the power-law exponent but the estimator needs an estimate of the minimal degree where the power law starts to hold and such an estimate typically requires careful empirical analysis.
Furthermore, power-law relations refer mostly to nodes of high degrees, of which there are only few, resulting in high variance 
in estimation (and thus low credibility). The maximum likelihood method
automatically weighs the information contained in the different degrees, and does so in 
an optimal way, as we show below. Of course, a drawback is that
it assumes that the PA mechanism is a good fit to the empirical network.
 
The main challenge in analyzing the maximum likelihood estimator of $\delta$ is that the model is 
a non-stationary Markov chain, which continuously visits new states, given by the growing network.
This requires careful analysis, but in the end we find that the quantities that are relevant for the estimation of
$\delta$ stabilize, as the network grows, leading to the result that the MLE
for $\delta$ based on observing the history of the network until time 
$n$ tends to a normal distribution of minimal variance when centered at the true $\delta$ and scaled by $\sqrt n$.
  
In practice it may well be that only a  snapshot of the network at time $n$ is observed, and not
its history through times $t=1,\ldots, n$.  A discovery that is surprising 
at first is that given fixed initial degrees $m_t=m$ this makes no difference for the
maximum likelihood procedure. It turns out that in this case,
which we refer to as the \textit{preferential attachment network with fixed initial degrees},  the snapshot at time $n$ is 
statistically sufficient for the full history. As a consequence the MLE
based on the snapshot is asymptotically normal after centering at $\delta$ and scaling by $\sqrt n$ in this case.

On the other hand, random initial degrees $m_t$ significantly complicate estimation when
observing only a snapshot of the network. Performing maximum likelihood 
would require either marginalizing the likelihood over all possible histories or implementing an iterative
approximation, e.g.\ of EM type. This seems computationally daunting. We also do not present
a full theoretical analysis. In fact, we conjecture that if the initial-degree 
distribution were also unknown no estimator for $\delta$ based only 
on the degrees in the network at time $n$ can attain a $\sqrt n$-rate of 
estimation, thus suggesting that the statistics change significantly. On the positive
side, we propose a ``quasi maximum likelihood'' estimator of $\delta$ that requires to know the initial degree
distribution but relies only on the final snapshot. We show this estimator to attain a $\sqrt n$-rate and to be asymptotically
normal, with a somewhat larger  variance than the MLE based on the complete evolution of the network.

The paper is organized as follows.  In Section~\ref{SectionMLE} we introduce basic notation and
derive the likelihood and the maximum likelihood estimator. In Section~\ref{sec:consistency} we
prove the consistency of this estimator.  In Section~\ref{SectionAN} we derive the main
result of the paper, which is the asymptotic normality of the maximum likelihood estimator,
by an application of the martingale central limit theorem.  Section~\ref{sec:fixed-initial-degree} gives a special case of the general results for fixed initial degrees.
Section~\ref{sec:quasi-mle} defines the quasi-maximum-likelihood estimator $\tilde \delta_n$, which does not depend on the history of the network, and establishes its asymptotics. 
Last but not the least, we present simulations in Section~\ref{sec-numerical} to illustrate the results.  

\section{Construction of the MLE}
\label{SectionMLE}
We start by introducing the affine preferential attachment model with random initial degree distribution.  Here we adapt the notation from
\cite{hofstadcomplexnet,deijfen2009preferential}.  Let $(m_t)_{t\ge 1}$ be an independent and
identically distributed (\textsc{iid}) sequence of positive integer-valued random variables.
The model produces a network sequence $\{ \PA_t(\delta)\}_{t=1}^\infty$, where for every $t$
the network $\PA_t(\delta)$ has a set $V_t = \{ v_0,v_1,\dots, v_{t} \}$ of $t+1$ vertices and $\sum_{i=1}^t m_i$ edges. 
The first network $\PA_1(\delta)$ consists of two
vertices $v_0$ and $v_1$ with $m_1$ edges between them.  For $t \ge 2$, given $\PA_{t-1}(\delta)$, a
new vertex $v_t$ is added with $m_t$ edges connecting $v_t$ to $V_{t-1}$, determined by the \textit{intermediate
  updating} preferential attachment rule.  This updating rule means that the edges are added sequentially using the
preferential attachment rule.  Define $\PA_{t,0}(\delta)= \PA_{t-1}(\delta)$ to be network after $v_{t-1}$ has
been fully integrated, and let $\PA_{t,1}(\delta), \PA_{t,2}(\delta),\ldots, \PA_{t,m_t}(\delta)$ be intermediate networks,
which add $v_t$ and its $m_t$ edges sequentially to $\PA_{t,0}(\delta)$, as follows.  For $1\le i\le m_t$, the network
$\PA_{t,i}(\delta)$ is constructed from $\PA_{t,i-1}(\delta)$ by adding an (additional) edge between $v_t$ 
and a randomly-selected vertex among $\{ v_0, v_1, \dots, v_{t-1} \}$. The probability that this is vertex $v_j$ is
proportional to $k+\delta$, if $v_j\in V_{t-1}$ has degree $k$ in $\PA_{t,i-1}(\delta)$. Here $\delta >-1$
is an unknown parameter.  The random choice is made through a multinomial trial on all the vertices in $V_{t-1}$.  In other words,
the conditional probability that the $i$-th edge of $v_t$ connects it to $v_j$ is
\begin{equation}
  \bbP\big( v_{t,i} \rightarrow v_j | \PA_{t,i-1}(\delta) \big) = \frac{\Deg_{t,i-1}(v_j) + \delta }{\sum_{v\in V_{t-1}} (\Deg_{t,i-1}(v) +  \delta)},
    \label{eqn-pa-evolution}
\end{equation}
where $\Deg_{t,i-1}(v)$ is the degree of $v$ in $\PA_{t,i-1}(\delta)$.  After all $m_t$ edges have been added to $v_t$, 
the network is given by $\PA_{t,m_t}(\delta) = \PA_{t}(\delta) = \PA_{t+1,0}(\delta)$.  

We define \(N_k(t)\) to be the number of vertices of degree $k$ in the network $\PA_t(\delta)$ (counting also $v_t$),
and $N_k(t,i-1)$ to be the number of vertices of degree $k$ in the network $\PA_{t,i-1}(\delta)$,  \emph{not counting} $v_t$ (so
belonging to $V_{t-1}$), for $1 \le i\le m_t$. By convention $N_k(t,0) = N_k(t-1)$.  
We denote by $D_{t,i}$ the degree of the vertex that was chosen when 
constructing $\PA_{t,i}(\delta)$ from $\PA_{t,i-1}(\delta)$.
So we can say that \emph{the $i$-th edge chosen by the vertex $v_t$ possesses degree $D_{t,i}$}.  
For the evolution of the number of vertices of degree \(k\), there are several scenarios.  
For given natural numbers $k$ and $i \le m_t$:
\begin{itemize}
    \item If $D_{t,i} \notin \lrbkt{k, k-1}$, then the number of vertices of degree $k$ remains unchanged, i.e.\ $N_k(t,i) = N_k(t,i-1)$.
    \item If $D_{t,i} = k-1$, then the vertex that was picked by the incoming vertex gets one extra connection and there is one more vertex of degree $k$, i.e.\ $N_k(t,i) = N_k(t,i-1) + 1$.    
    \item If $D_{t,i} = k$, then the vertex that was picked by the incoming vertex becomes a vertex of degree $k+1$ and there is one fewer vertex of degree $k$, i.e.\ $N_k(t,i) = N_k(t,i-1) -1$.  
\end{itemize}
After the last update $i = m_t$, the vertex $v_t$ is fully integrated in the network.
Since its degree is $m_t$, we have $N_k(t) = N_k(t,m_t) + 1_{\lrbkt{ k = m_t}}$.  
This concludes the time step $t$ and the total number of vertices of the network becomes $t+1$. 

These observations are summarised in the following equation for the degree evolution:
\begin{equation}
    N_k(t) = N_k(t-1) + \sum_{i=1}^{m_t} 1_{\lrbkt{ D_{t,i} = k-1}} - \sum_{i=1}^{m_t} 1_{\lrbkt{D_{t,i} = k}} + 1_{\lrbkt{k=m_t}}.
    \label{eqn-Nkt-evo}
\end{equation}
The total number of edges in $\PA_t(\delta)$ is  $M_t := \sum_{j=1}^t m_j$.  Clearly the maximal degree in the network $\PA_t(\delta)$
is bounded by $M_t$, i.e.\ $N_k (t) = 0$ for any $k > M_t$. Furthermore, $\sum_{k=1}^{M_t} N_k(t) = t+1 $ is the total number
of vertices at time $t$.

The conditional probability of the incoming vertex choosing an existing vertex of degree $k$ to connect to is 
\begin{equation}
    \bbP\big(D_{t,i}=k| \PA_{t,i-1}(\delta), (m_t)_{t\ge 1}\big) = \frac{(k+\delta) N_k(t,i-1)}{ \sum_{j=1}^\infty (j + \delta) N_j(t,i-1)}.
    \label{eqn-Qkt}
\end{equation}
The denominator in this expression counts the ``total preference'' of all the vertices,  and can be written as
\begin{align*}
S_{t,i-1}(\delta) &=   \sum_{j=1}^\infty (j+\delta) N_j(t,i-1) = \sum_{v \in V_{t-1}} (\Deg_{t,i-1}(v) +\delta) \\
&= t \delta + 2M_{t-1} + (i-1). 
\end{align*}
Abbreviate the degree sequence at time $t$ by $D_t:=(D_{t,1},\dots,D_{t,m_t})$.  
The conditional  likelihood of observing $(D_t)_{t=2}^{n} = (d_t)_{t=2}^{n}$ given the edge counts $m_1,m_2,\ldots$ is
\begin{equation}
    \begin{aligned}
\bbP((D_t)_{t=2}^{n} = (d_t)_{t=2}^{n}\mid (m_t)_{t\ge 1} ) = \prod_{t=2}^{n}\prod_{i=1}^{m_t} \frac{(d_{t,i}+\delta)N_{d_{t,i}}(t,i-1)}{ S_{t,i-1}(\delta) }.
    \end{aligned}
    \label{eqn-full-lhd}
\end{equation}
We are interested in estimating $\delta$ and assume that the distribution of the edge counts $m_t$
does not contain information on this parameter. If the edge counts are observed, we can condition on them throughout,
and treat the preceding as the full likelihood of the observation $(D_t)_{t\ge 2}$. In view of \eqref{eqn-pa-evolution}
the likelihood of observing the full evolution of the network up to $\PA_n(\delta)$ is a function
of $(D_t)_{t=2}^{n}$ and hence the latter vector is statistically sufficient for this full evolution,
given the edge counts $(m_t)$. In the following we shall see that actually 
observing only the ``snapshot'' of the network $\PA_n(\delta)$ 
at time $n$ is already statistically sufficient for $\delta$, given the edge counts $(m_t)$.

Define $N_{>k}(t)$ to be the number of vertices in $\PA_t(\delta)$ of degree (strictly) bigger than $k$, 
i.e.\ $N_{>k}(t)  = \sum_{j = k+1}^{M_t} N_j(t)$.  As first observed in \cite{empiestimator}, we have the following lemma.
(In \cite{empiestimator} the l.h.s.\ of  \eqref{eqn-history-present} is called $N_{\rightarrow k}(t)$.) 
\begin{lemma}
The number of vertices with degree strictly bigger than $k$ is equal to the number of times a vertex of degree 
$k$ was chosen by the incoming vertices until and including time $n$ 
plus the number of vertices with initial degrees strictly bigger than $k$.  In other words,
if $R_{>k}(n) =2\cdot 1_{\{ m_1>k\}} + \sum_{t=2}^{n} 1_{\lrbkt{m_t >k}} $, then 
    \begin{equation}
        \sum_{t=2}^{n} \sum_{i=1}^{m_t} 1_{\{D_{t,i} = k\}} = N_{>k}(n) - R_{>k}(n).
        \label{eqn-history-present}
    \end{equation}
    \label{lemma-crucial-observation}
\end{lemma}

We introduce the shorthand $D^{(n)} = (D_t)_{t=2}^{n}$, and from \eqref{eqn-full-lhd} have the log-likelihood function  
\begin{equation*}
\begin{aligned}
l_n(\delta| D^{(n)})  &=  \sum_{t=2}^{n}\sum_{i=1}^{m_t}\Bigl[\log N_{D_{t,i}}(t, i-1) +   \log(D_{t,i}+\delta)  -  \log S_{t,i-1}(\delta) \Bigr]\\
        &=  \sum_{t=2}^{n}\sum_{i=1}^{m_t} \log N_{D_{t,i}}(t,i-1) + \sum_{k=1}^{\infty}\log(k+\delta) (N_{>k}(n)
- R_{>k}(n))\\ 
&\qquad\qquad\qquad\qquad\qquad- \sum_{t=2}^n \sum_{i=1}^{m_t} \log S_{t,i-1}(\delta). 
   \end{aligned}
\end{equation*}
where the second equality comes from applying Lemma~\ref{lemma-crucial-observation} and the fact that 
\[
    \sum_{t=2}^{n}\sum_{i=1}^{m_t} \log(D_{t,i} +\delta) = \sum_{k=1}^{\infty} \Bigl[\log (k+\delta) \sum_{t=2}^{n}\sum_{i=1}^{m_t} 1_{\{ D_{t,i} = k\} }\Bigr].
\]  
It follows that the likelihood factorises in a part not involving $\delta$ and a part involving $\delta$ and the variables $N_{>k}(n)$, given the edge counts $(m_t)$. Thus by the factorization theorem 
(see \cite{Lehmann}, Corollary~2.6.1) the vector $\bigl(N_{>k}(n)\bigr){}_{k\ge 1}$ is  statistically sufficient for $\delta$, given $(m_t)$.
This vector is completely determined by the network at time $n$. In particular, observing the network only at time
$n$ is sufficient for $\delta$ relative to observing its evolution up to and including time $n$.

For inference on $\delta$ we can drop the first term of the log likelihood, which does not depend on $\delta$,
and normalise the remaining part by $n+1$ (note that there are $n+1$ vertices in the network at time $n$).
We take the parameter space for $\delta$ to be $[-a,b]$, for given numbers $-1<-a< b<\infty$.
The maximum likelihood estimator (MLE) of $\delta$  is    then given by
$\hat{\delta}_n  =   \argmax_{\delta \in [-a,b]} \iota_n(\delta)$, for
\begin{equation}
    \begin{aligned}
\iota_n(\delta)=\sum_{k=1}^{\infty}\log(k+\delta) \frac{N_{>k}(n)- R_{>k}(n)}{n+1}- \frac{1}{n+1} \sum_{t=2}^{n}\sum_{i=1}^{m_t} \log S_{t,i-1}(\delta). 
    \end{aligned}
    \label{eqn-iota}
\end{equation}
Provided that the maximum is taken in the interior of the parameter set,
the MLE is a solution of the likelihood equation $\iota_n'(\delta)=0$.  This derivative is given by 
\begin{equation}
    \iota'_n(\delta) = \sum_{k=1}^\infty \frac{1}{k+\delta} \frac{N_{>k}(n) - R_{>k}(n)}{n+1} - \frac{1}{n+1}\sum_{t=2}^n \sum_{i=1}^{m_t} \frac{t}{S_{t,i-1}(\delta)}. 
    \label{eqn-iota-def}
\end{equation}

\section{Consistency}
\label{sec:consistency}
The empirical degree distribution in $\PA_n(\delta)$ is defined as 
$$p_k(n)=\frac{N_k(n)}{n+1}.$$
In \cite{deijfen2009preferential} it is shown that this distribution tends to a limit as $n\ra\infty$.
Let $(r_k)_{k\ge 1}$ be the probability distribution of the initial degree (and the number of edges added in every step), i.e.\
\begin{equation}
    r_k = \bbP(m_1 = k), \quad k \ge 1.
    \label{eqn-rk-rid}
\end{equation}
Assume that this distribution has finite mean $\mu \ge 1$ and finite second moment $\mu^{(2)}$,
and write the shorthand $\theta = 2 + \delta/\mu$.  Then the limiting degree distribution $(p_k)_{k\ge 1}$ 
satisfies the recurrence relation
\begin{equation}
    p_k = \frac{k-1+\delta}{ \theta} p_{k-1} - \frac{k+\delta}{\theta} p_k + r_k,\qquad k\ge 1.
    \label{eqn-pk-recursion-rid}
\end{equation}
Starting from the initial value $p_0 = 0$, we can solve the recurrence relation by 
\begin{equation}
    p_k = \frac{\theta}{k+\delta + \theta} \sum_{i=0}^{k-1} r_{k-i} \prod_{j=1}^i \frac{k-j + \delta}{ k-j+\delta+\theta}, \quad k\ge 1, 
    \label{eqn-pk-rid}
\end{equation}
where the empty product is defined to be $1$, should it arise.  
Because $\sum_{k\ge 1} p_k = \sum_{k \ge 1} r_k = 1$, 
in view of the recurrence relation \eqref{eqn-pk-recursion-rid}, the probabilities $(p_k)_{k\ge 1}$ define a proper probability distribution. 

We list two results on the limiting degree distribution of the preferential attachment model with random initial degrees.  
See \cite{deijfen2009preferential} for proofs.
 
\begin{proposition}
  If the initial degrees $(m_t)_{t=1}^\infty$ have finite moment of order $1+\varepsilon$ for some $\varepsilon>0$, then there exists a
constant $\gamma\in(0,1/2)$ such that
    \begin{equation*}
        \lim_{n \rightarrow \infty} \bbP\Bigl( \max_{k\ge 1} |p_k(n) - p_k| \ge n^{-\gamma}\Bigr) = 0,
    \end{equation*}
    where $(p_k)_{k=1}^\infty$ is defined as in \eqref{eqn-pk-rid}. 
    \label{prop-degree-distribution}
\end{proposition}

In the case that the initial degree is degenerate, i.e.\ $r_m = 1$ for some integer $m \ge 1$, the rate
of convergence in this result can be improved, and the limiting degree distribution takes a simpler form, as follows.

\begin{proposition}
    If $r_m = 1$ for some integer $m \ge 1$, then there exists a constant $C>0$ such that
    \begin{equation*}
        \lim_{n \rightarrow \infty} \bbP\Bigl( \max_{k\ge 1} |p_k(n) - p_k| \ge C\sqrt{\log n /n} \Bigr) = 0,
    \end{equation*}
    where $(p_k)_{k=1}^\infty$ is defined as follows: 
    \begin{equation}
p_k = \begin{cases} 0,& \text{ if } k<m,\\
\frac{\theta \Gamma(k+\delta) \Gamma(m+\delta+\theta)}{\Gamma(m+\delta) \Gamma(k+1+\delta+\theta)},& \text{ if } k\ge m.
\end{cases}
        \label{eqn-degree-fixed-m}
    \end{equation}
Furthermore, if $m=1$, so that $r_1 = 1$, then the empirical degree 
$p_k(n)$ converges also almost surely to $p_k$, as $n \rightarrow \infty$, for every $k$.  
\end{proposition}

Next we give a lemma that will be essential to our analysis later.  
For a summable sequence $(a_k)_{k\ge1}$, write $a_{>k}=\sum_{j>k}a_j$.

\begin{lemma}
The following recurrence relation holds, with  $\theta = 2+ \delta/\mu$,
    \begin{equation}
        p_{>k} = \frac{k+\delta}{ \theta} p_k + r_{>k}.
        \label{eqn-pbk-pk-rk}
    \end{equation}
    \label{lemma-pbk-pk-rk}
\end{lemma}

\begin{proof}
We simply sum up terms of \eqref{eqn-pk-recursion-rid} and cancel repeated terms. 
\end{proof}

From now on we shall put a superscript $^{(0)}$ to stress that we consider limiting distributions under the true
value $\delta_0$ of the parameter.  In view of Proposition~\ref{prop-degree-distribution}, 
$N_{>k}(n)/(n+1)$ is asymptotic to $p_{>k}^{(0)}$, while by the Law of Large Numbers
$R_{>k}(n)/(n+1)$ tends to $r_{>k}$. Furthermore, for fixed $i$ the sequence $S_{t, i-1}(\delta)/t=\delta+2\bar m_{t-1}+(i-1)/t$  is asymptotic to 
$\delta+2\mu$, again by the Law of Large Numbers. Therefore, we expect the criterion  $\iota'_n(\delta)$ given in 
\eqref{eqn-iota-def} to be
asymptotic to 
\begin{equation}
    \iota'(\delta) = \sum_{k=1}^\infty \frac{p_{>k}^{(0)} - r_{>k}}{k+\delta} - \frac{1}{2+\delta/\mu}.
\label{EqDerivativeIota}
\end{equation}
Consequently, we expect that the MLE $\hat\delta_n$ will be asymptotic to the
solution of the equation $\iota'(\delta)=0$. 
Because of \eqref{eqn-pbk-pk-rk},
$$\iota'(\delta_0) = \sum_{k=1}^\infty\frac{p_{>k}^{(0)} - r_{>k}} {k+\delta_0} - \frac{1}{2+\delta_0/\mu} = 0.$$
Thus the true parameter is indeed a solution to this equation.  
The following lemmas show that this solution is unique.

Define
\begin{equation}
q_k = \frac{p_{>k} - r_{>k}}{\mu} = \frac{(k+\delta)p_k}{2\mu + \delta}.
\label{EqDefqk}
\end{equation}

\begin{lemma}
For any nonnegative sequence  $(v_k)_{k=1}^\infty$
that is strictly decreasing with respect to $k$ and any $\delta_1 > \delta_2$, we have,
for $q_k(\delta)$ given in \eqref{EqDefqk} (where $p_k=p_k(\delta)$ as well),
    \begin{equation*}
        \sum_{k=1}^\infty q_k(\delta_1) v_k > \sum_{k=1}^\infty q_k(\delta_2)v_k.
    \end{equation*}
    \label{lemma-qd1-qd2-compare}
\end{lemma}

\begin{proof}
We first show that $\sum_{k\ge 1} q_k = 1$. By manipulating the recurrence relation \eqref{eqn-pk-recursion-rid}, we find 
    \begin{align*}
        \sum_{k=1}^\infty q_k & = \frac{1}{2\mu + \delta}\sum_{k=1} (k+\delta) p_k=\frac{1}{2\mu + \delta}\Bigl( \sum_{k=1}^\infty kp_k + \delta \Bigr) \\
        & = \frac{1}{2\mu + \delta} \left( \sum_{k=1}^\infty k \Big( \frac{k-1+\delta}{ 2+ \delta/\mu} p_{k-1} - \frac{k + \delta}{ 2 + \delta/\mu} p_k + r_k \Big) + \delta\right) \\
        & = \frac{1}{2\mu + \delta} \left( \sum_{k=1}^\infty kr_k + \sum_{k = 1}^\infty \frac{k+\delta}{ 2+ \delta/\mu} p_k + \delta\right) \\
        & = \frac{1}{2\mu + \delta} \Bigl( \mu + \mu\sum_{k=1}^\infty q_k + \delta\Bigr).
    \end{align*}
The only solution to this equation for  $\sum_k q_k$ has $\sum_k q_k =1$. 

By the recurrence formula \eqref{eqn-pbk-pk-rk} and \eqref{eqn-pk-recursion-rid}, we have $q_k = (k+\delta)p_k /(2\mu +\delta)$,
 and $q_{k+1} = (k+1+\delta)(q_k + r_{k+1}/\mu)/(k+1+\delta+2+\delta/\mu)$.  Therefore the 
derivative $u_k(\delta) = \frac{d}{d \delta} q_k(\delta) $ satisfies   
    \begin{equation*}
        u_{k+1}(\delta) = \frac{2 -(k+1)/\mu}{ (k+1+\delta+2+\delta/\mu)^2}(q_k(\delta) + r_{k+1}/\mu) + \frac{k+1+\delta}{ k+ 1 + \delta +2+ \delta/\mu} u_k(\delta). 
    \end{equation*}
The initial value of this sequence is positive, since
    \begin{equation*}
        u_1(\delta) = \frac{d}{d\delta} q_1(\delta) = \frac{2\mu - 1}{ \mu^2(1+ \delta+ 2 + \delta/\mu)^2}r_1 >0.
    \end{equation*}
From the recursion it follows that $u_{k+1}(\delta)$ remains positive at least as long as $k+1\le 2\mu$. 
For $k+1>2\mu$ the first term of the recursion is negative, while the second term has the sign of $u_k(\delta)$. From the fact that $\sum_k q_k(\delta)=1$ for every $\delta$, it follows that $\sum_k u_k(\delta)=0$, and hence
$u_k(\delta)$ cannot remain positive indefinitely. If $K(\delta)+1$ is the first $k$ for which $u_k(\delta)<0$, then
it must be that $K(\delta)+1>2\mu$, which implies that $u_k(\delta)<0$ for every $k>K(\delta)+1$ as well.
Since $v_k$ is decreasing it follows that
$$\sum_k v_k u_k(\delta)>\sum_{k\le K(\delta)}v_{K(\delta)}u_k(\delta) +\sum_{k> K(\delta)}v_{K(\delta)}u_k(\delta) 
= v_{K(\delta)}0=0.$$
Integrating this over the interval $[\delta_2,\delta_1]$  gives the assertion.
\end{proof}

\begin{lemma}
The function $\delta\mapsto\iota'(\delta)$ possesses a unique zero at $\delta=\delta_0$.
It is positive for $\delta<\delta_0$ and negative if $\delta>\delta_0$.
    \label{lemma-uniq-zero-mle}
\end{lemma}

\begin{proof}
Following the definition \eqref{EqDerivativeIota} of $\iota'$ it was seen that $\delta_0$ is a zero.
Fix some $\delta \neq \delta_0$. Since $1= \sum_{k} p_k(\delta)$ and $q_k(\delta) = (k+\delta)p_k(\delta)/(2\mu+\delta)$, 
we can rewrite $\iota'(\delta)$ as 
    \begin{align*}
    \iota'(\delta) & = \sum_{k=1}^\infty \frac{\mu q_k^{(0)}}{k+\delta} - \frac{1}{2+\delta/\mu} \\
    & =  \sum_{k=1}^\infty \frac{ \mu q_k^{(0)}}{k+\delta} - \sum_{k=1}^\infty \frac{(k+\delta)p_k(\delta)}{(k+\delta)(2+\delta/\mu)}\\
    & =  \sum_{k=1}^\infty \frac{ \mu q_k^{(0)}}{k+\delta} - \sum_{k=1}^\infty \frac{\mu q_k(\delta)}{k+\delta}.
    \end{align*}
Applying Lemma~\ref{lemma-qd1-qd2-compare} with $v_k = 1/(k+\delta)$,
we see that $\iota'(\delta) > 0$ when $\delta < \delta_0$, and $ \iota'(\delta) < 0$ when $\delta > \delta_0$.
\end{proof}

The proof of consistency of the MLE will be based on uniform convergence of $\iota'_n$ to $\iota'$, 
together with the uniqueness of the zero of $\iota'$. For the convergence, and also for the proof of
asymptotic normality, we need the following lemma.

\begin{lemma}[Cesàro convergence for random variables]
Let $(X_t)_{t\in\NN}$ be a sequence of random variables, $(a_t)_{t\in\NN}$ a sequence of numbers, $X$ and $a$ a random variable and
number, and let $\overline{X}_t$ and $\overline{a}_t$ be the average of the first $t$ variables or numbers, respectively.
\begin{enumerate}[i).]
\item If \(X_t \xrightarrow{\text{a.s.}} X \), then $\overline{X}_t \xrightarrow{\text{a.s.}} X$.
\item If \(X_t \xrightarrow{L_1} X \), or equivalently  \( X_t \xrightarrow{\text{P}} X\) and $(X_t)_{t\in\NN}$ is uniformly integrable, 
then $\overline{X}_t \xrightarrow{L_1} X$.
\item If \(X_t \xrightarrow{L_1} X \) and $\overline{a}_t \rightarrow a$ and $\overline{|a|}_t=O(1)$, then $\overline{(aX)}_t \xrightarrow{L_1} a X$.
    \end{enumerate}
    \label{lem-cesaro-mean}
\end{lemma}

\begin{proof}
Statement (i) is the usual Ces\`aro convergence, applied to almost every of the deterministic sequences $X_t(\omega)$ obtained
for elements $\omega$ of the underlying probability space.

Statement (ii) is the special case of (iii) with $a_t=1$, for every $t$.

To prove statement (iii) we decompose
\begin{align*}
   |\overline{(aX)}_t - a X| & = \Bigl| \frac{1}{t} \sum_{i=1}^t a_i X_i - \frac{1}{t} \sum_{i=1}^t a_i X + \frac{1}{t}  \sum_{i=1}^t a_i X - a X \Bigr| \\
        & \le  \frac{1}{t} \sum_{i=1}^K \bigl|a_i(X_i -X) \bigr| + \frac{1}{t}  \sum_{i=K+1}^t \bigl|a_i(X_i - X) \bigr| + |X| |\overline{a}_t - a|.
\end{align*}
Take the expectation across to bound the expected value of the left side by
$$\frac{K}{t}\max_{1\le i\le K}|a_i|\max_{t}\bbE |X_t-X|+  \overline{|a|}_t \max_{K<i\le t} \bbE|X_i - X|+
|\overline{a}_t - a |\,\bbE |X|.$$
Because $\bbE|X_i-X|\ra0$ as $i\ra\infty$, for any $\varepsilon>0$, there exists $K$ such that $\sup_{i>K} \bbE |X_i - X| < \varepsilon $.  
Then the second term is bounded above by a constant times $\varepsilon$, by the assumption on $\overline{|a|}_t$.
For fixed $K$ the first and third terms tend to zero as $t\ra\infty$. 
Thus the limsup as $t\ra\infty$ of the whole expression is bounded by a multiple of $\varepsilon$, for every $\varepsilon>0$.
\end{proof}

\begin{lemma}
The derivative $\iota'_n$ of the log-likelihood function converges uniformly to the limiting criterion $\iota'$, i.e.\ as $n \rightarrow \infty$
for every $\epsilon>0$,
    \begin{equation*}
        \sup_{\delta> -1+\epsilon } |\iota'_n - \iota' |(\delta) \xrightarrow{P} 0. 
    \end{equation*}
    \label{lemma-loglhd-uniform-conv} 
\end{lemma}

\begin{proof}
For $r_{>k}(n) = R_{>k}(n)/(n+1)$, the difference $\iota'_n(\delta) - \iota'(\delta)$ can be decomposed as
\begin{equation}
 \begin{aligned}
 & \sum_{k=1}^\infty \frac{p_{>k}(n) - p_{>k}^{(0)}}{k+\delta}  + \sum_{k=1}^\infty \frac{r_{>k}(n) - r_{>k}}{k+\delta}   
   -  \frac{1}{n+1}\sum_{t=2}^n\sum_{i=1}^{m_t} \frac{t}{S_{t,i-1}(\delta)} + \frac{\mu}{2\mu + \delta}. 
\end{aligned}
\label{eqn-separate-terms-iota-difference}
\end{equation}
We deal with the first two terms and the difference of the last two terms separately.  

As $kN_{>k}(n) \le  2M_n$, where
$M_n=\sum_{t=1}^nm_t$ is the total number of edges in $\PA_n(\delta)$, we have $p_{>k}(n) \le 2\overline{m}_n/k$, for every $k$.  Hence,
for $\delta\ge -\eta:=-1+\epsilon$,
 \begin{align*}
        \sum_{k=1}^\infty \frac{|p_{>k}(n) - p_{>k}^{(0)}|}{k+\delta} 
        & \le \sum_{k\le K} \frac{|p_{>k}(n) - p_{>k}^{(0)}|}{ k-\eta} + \sum_{k>K} \frac{2\overline{m}_n}{k(k-\eta)} 
+ \sum_{k>K} \frac{p_{>k}^{(0)}}{k-\eta}.
    \end{align*}
Since $\overline{m}_n \rightarrow \mu$ almost surely, by the Law of Large Numbers,
the second term on the right side can be made arbitrarily small by choice of $K$.
The same is true for the third term as $p^{(0)}_{k}$ follows a power law with exponent bigger than $2$.  
For any fixed $K$ the first term converges in probability to $0$ as $n\ra\infty$, by Proposition~\ref{prop-degree-distribution}.
Thus the full expression tends to zero.
    
The variable $r_{>k}(n) - r_{>k}$ can be written in the form $2(1_{\lrbkt{m_1>k}}-r_{>k})/(n+1)+\sum_{t=2}^n(1_{\lrbkt{m_t>k}} -r_{>k})/(n+1)$.
This is a weighted sum of independent centered Bernoulli variables with success probability $r_{>k}$. Its first absolute
moment can be bounded by its standard deviation and is bounded by a multiple of the root of $r_{>k}(1-r_{>k})/(n+1)$.
It follows that the supremum over $\delta>-\eta= -1+\epsilon$ of the absolute value of the second term has expected value bounded above by
a multiple of 
$$\frac 1{\sqrt{n+1}}\sum_{k=1}^\infty\frac{\sqrt{r_{>k}(1-r_{>k})}}{k-\eta}.$$
Since $r_{>k}\le k^{-1}\mu$, by Markov's inequality, the series converges easily, and the expression tends to zero as $n\rightarrow\infty$.
    
With slight abuse of notation write $\overline{m}_n=\sum_{t=1}^nm_t/(n+1)$. The third term can be decomposed as
\begin{align*}
& -\frac{1}{n+1}\sum_{t=2}^n \sum_{i=1}^{m_t}\Bigl[ \frac{1}{S_{t,i-1}(\delta)/t} -\frac{1}{\delta+2\overline{m}_{t-1} }\Bigr]\\
&\qquad -  \frac{1}{n+1} \sum_{t=2}^n \Bigl[\frac{ m_t}{\delta+2\overline{m}_{t-1}} - \frac{m_t}{2\mu + \delta} \Bigr] 
- \Bigl[ \frac{1}{n+1}  \sum_{t=2}^n \frac{m_t}{2\mu + \delta} -\frac{\mu}{ 2\mu+ \delta} \Bigr]\\
&= - \frac{1}{n+1}\sum_{t=2}^n \sum_{i=1}^{m_t} \frac{(i-1)/t}{(\delta+2\overline{m}_{t-1}+(i-1)/t)(\delta+2\overline{m}_{t-1})}\\
&\qquad -  \frac{1}{n+1} \sum_{t=2}^n \frac{ m_t(2\mu-2\overline{m}_{t-1})}{(\delta+2\overline{m}_{t-1})(2\mu + \delta)}
-  \frac{\overline{m}_n-\mu}{2\mu + \delta}.
\end{align*}
The supremum over $\delta>-\eta$ of the absolute value of this expression is bounded
above by
$$\frac{1}{n+1}  \sum_{t=2}^n  \frac{m_t^2/t}{(2\overline{m}_{t-1}-\eta)^2}
+ \frac{1}{n+1} \sum_{t=2}^n \frac{ m_t2|\mu-\overline{m}_{t-1}|}{(2\overline{m}_{t-1}-\eta)(2\mu-\eta)}
+  \frac{|\overline{m}_n-\mu|}{2\mu -\eta }.$$
The third term tends to zero almost surely by the Law of Large Numbers.
In the first term we have that the variables $X_t:=t^{-1}/(2\overline{m}_{t-1}-\eta)^2$ converge almost surely to 0
as $t\rightarrow\infty$, while the averages of 
the variables $a_t:=m_t^2$ tend to $\mu^{(2)}$ almost surely, again by the Law of Large Numbers. Applying  Lemma~\ref{lem-cesaro-mean} to the
sequences of \emph{numbers} $X_t(\omega)$ and $a_t(\omega)$ obtained by selecting $\omega$ from the
underlying probability space so that both convergences are valid, we see that the first term tends
to zero, for such $\omega$, and hence almost surely.
The second term tends to zero by the same argument, now with the choice
$X_t:=2|\mu-\overline{m}_{t-1}|/\bigl((2\overline{m}_{t-1}-\eta)(2\mu-\eta)\bigr)$.
\end{proof}

Combining the preceding Lemma~\ref{lemma-uniq-zero-mle} and Lemma~\ref{lemma-loglhd-uniform-conv} gives the following theorem.

\begin{theorem}
\label{prop-consistency}
The MLE $\hat \delta_n$ is consistent: $\hat\delta_n\rightarrow\delta_0$, in probability under $\delta_0$, for every
$\delta_0\in (-a,b)$.
\end{theorem}

\begin{proof}
Because $\iota'$ is continuous on $[-a,b]$ and vanishes only at $\delta_0$, we have that
$\inf_{\delta\in [-a,b]: |\delta-\delta_0|>\epsilon}|\iota'(\delta)|>0$, for every $\epsilon$. 
More precisely, by Lemma~\ref{lemma-uniq-zero-mle} it is bounded away from zero in the positive direction for $\delta<\delta_0-\epsilon$ and
in the negative direction if $\delta>\delta_0+\epsilon$. Since $\iota_n'$ tends uniformly
to $\iota'$, by Lemma~\ref{lemma-loglhd-uniform-conv}, the same is true for $\iota_n$, with probability tending to one. 
This shows that the maximum of $\iota_n$ must be contained in $[\delta_0-\epsilon,\delta_0+\epsilon]$,
with probability tending to one. 
\end{proof}

\section{Asymptotic Normality}
\label{SectionAN}
We shall apply the following martingale central limit theorem 
(see Corollary~$3.1$ in \cite{hall2014martingale} or Theorem~XIII.1.1 in \cite{pollard}) to study the asymptotic normality of the MLE.
The triangular array version with $k_n\rightarrow\infty$ given here is equivalent to the theorem in the latter
reference (stated for $k_n=n$), as remarked preceding its statement on page~171.

\begin{proposition}
Suppose that for every $n\in\NN$ and $k_n\ra\infty$ the random variables $X_{n,1},\ldots, X_{n,k_n}$ are  a martingale difference sequence 
relative to an arbitrary filtration $\mathcal{F}_{n,1}\subset \mathcal{F}_{n,2}\subset\cdots\subset \mathcal{F}_{n,k_n}$. 
If $\sum_{i=1}^{k_n} \bbE [ X_{n,i}^2 | \calF_{n,i-1}] \xrightarrow{P} v$ for a positive constant $v$, and 
$$\sum_{i=1}^{k_n} \bbE[X_{n,i}^2 1_{ \{ |X_{n,i}| > \varepsilon\} }| \calF_{n,i-1}] \xrightarrow[]{P} 0 ,$$
for every $\varepsilon >0$, then $\sum_{i=1}^{k_n}X_{n,i}\rightsquigarrow N(0,v)$.  
    \label{prop-martingale-clt}
\end{proposition}

\begin{lemma}
Given almost every sequence $(m_t)_{t=1}^\infty$
we have, under $\delta_0$, 
\begin{equation}
      \sqrt{n}\big(\iota'_n(\delta_0) - \iota'(\delta_0)\big) \rightsquigarrow N(0, \nu_0), 
        \label{eqn-clt-mle-z}
    \end{equation}
where $\iota'(\delta_0) =0$ and
$$\nu_0=\sum_{k=1}^\infty\frac{\mu q_k^{(0)}}{(k+\delta_0)^2}-\frac{\mu}{(2\mu+\delta_0)^2}.$$
    \label{prop-clt-mle-z}
\end{lemma} 

\begin{proof}
Throughout the proof we condition on $(m_t)_{t=1}^\infty$, without letting this show up
in the notation.

We can write
$$ \iota_n'(\delta_0) = \frac{1}{n+1}\sum_{t=2}^n\sum_{i=1}^{m_t} Y_{t,i},$$
for
$$Y_{t,i}=\frac{1}{D_{t,i}+\delta_0} - \frac{t}{ S_{t,i-1}(\delta_0)}=
\frac{1}{D_{t,i}+\delta_0} - \frac{1}{ \delta_0 +2\overline{m}_{t-1} + (i-1)/t}.  $$
As to be expected from the fact that they are score functions,
the variables $Y_{2,1},Y_{2,2},\ldots,Y_{2,m_2},Y_{3,1},\ldots,Y_{3,m_3},Y_{4,1}\ldots$ are martingale
differences relative to the filtration $\calF_{2,1}\subset\calF_{2,2}\subset\cdots\subset\calF_{2,m_2}\subset\calF_{3,1}\subset\cdots\subset\calF_{3,m_3}
\subset\calF_{4,1}\subset\cdots$  obtained by letting $\calF_{t,i}$ correspond to observing the evolution of the PA graph
up to $\PA_{t,i}(\delta)$. Indeed, in view of \eqref{eqn-Qkt},
$$\bbE[Y_{t,i} | \calF_{t,i-1}]=\sum_{k=1}^\infty \frac 1{k+\delta_0}\frac{N_k(t,i-1)(k+\delta_0)}{S_{t,i-1}(\delta_0)}-\frac t{S_{t,i-1}(\delta_0)}
=0,$$
since $\sum_k N_k(t,i-1)=t$ is the number of vertices in the graph at time $t$, for every $i$ (not counting $v_t$).
(Set $\calF_{t,0}=\calF_{t-1,m_{t-1}}$.)

We now apply Proposition~\ref{prop-martingale-clt} to the triangular array of martingale differences
$X_{2,1},\ldots, X_{2,m_2},\ldots, X_{n,m_n}$, for $n=1,2,\ldots$, and $X_{t,i}=Y_{t,i}/\sqrt{n+1}$. The $n$-th row possesses
$M_n=\sum_{t=2}^nm_t\ra\infty$ variables. Since the variables $Y_{t,i}$ are uniformly bounded by $2/(1+\delta_0)$, the
Lindeberg condition, in the display of Proposition~\ref{prop-martingale-clt}, is trivially satisfied. We need to show that
$$\frac1{n+1}\sum_{t=2}^n\sum_{i=1}^{m_t}\bbE[Y_{t,i}^2| \calF_{t,i-1}] \xrightarrow[]{P} \nu_0.$$
In view of \eqref{eqn-Qkt},
\begin{align*}\bbE[Y_{t,i}^2| \calF_{t,i-1}]
&=\bbE\Bigl[\frac1{(D_{t,i}+\delta_0)^2}| \calF_{t,i-1}\Bigr]-\Bigl(\frac t{S_{t,i-1}(\delta_0)}\Bigr)^2\\
&=\sum_{k=1}^\infty \frac 1{(k+\delta_0)^2}\frac{N_k(t,i-1) (k+\delta_0)}{S_{t,i-1}(\delta_0)}-\Bigl(\frac t{S_{t,i-1}(\delta_0)}\Bigr)^2.
\end{align*}
Since $i$ edges are added when constructing $\PA_{t,i}(\delta)$ from $\PA_{t,0}(\delta)$, the number of
nodes of degree $k$ cannot change by more than $i\le m_t$. Therefore, for every $t$,
$$\max_{1\le i\le m_t}\Bigl|\frac{N_k(t,i-1)}t-\frac{N_k(t,0)}t\Bigr|\le \frac{m_t}t.$$
Since $m_t$ has finite second moment, we have $\sum_t \bbP(m_t>t\epsilon)<\infty$, for every $\epsilon>0$,
and hence $m_t/t\rightarrow 0$, almost surely, as $t\ra\infty$. We combine
this with the preceding display and Proposition~\ref{prop-degree-distribution} 
to see that $N_k(t,i-1)/t\rightarrow p_k^{(0)}$ in probability, as $t\rightarrow\infty$, for every
fixed $k$, uniformly in $1\le i\le m_t$. As a function of $k$, the numbers 
$N_k(t,i-1)/t$ are a probability distribution on $\NN$, and hence 
$\sum_k |N_k(t,i-1)/t-p_k^{(0)}|\ra 0$, by Scheffe's theorem, uniformly in $1\le i\le m_t$.
In particular, the $N_k(t,i-1)/t$ are uniformly integrable (summable), whence by the dominated convergence
theorem also, uniformly in $1\le i\le m_t$, as $t\ra\infty$, 
$$\sum_k \Bigl|\frac{N_k(t,i-1)/t}{k+\delta_0}-\frac{p_k^{(0)}}{k+\delta_0}\Bigr|\xrightarrow[]{P} 0.$$
By the definition of $S_{t,i-1}(\delta_0)$, we also have 
$$\max_{1\le i\le m_t}\Bigl|\frac{S_{t,i-1}(\delta_0)}t-(\delta+2\overline{m}_{t-1})\Bigr|\le \frac{2m_t}t.$$
Therefore, by the Law of Large Numbers we obtain that 
$S_{t,i-1}(\delta_0)/t\ra (\delta_0+2\mu)$, almost surely, uniformly in $1\le i\le m_t$.

Combining the preceding we see that, for almost every sequence $(m_t)$, as $t\ra\infty$,
$$\frac1{m_t}\sum_{i=1}^{m_t}\sum_k \frac{N_k(t,i-1)}{(k+\delta_0)S_{t,i-1}(\delta_0)}\xrightarrow[]{P}
\sum_k \frac{p_k^{(0)}}{(k+\delta_0) (\delta_0+2\mu)}.$$
Next by Lemma~\ref{lem-cesaro-mean}, applied with $X_t$ equal to the left side of the preceding display (which is
bounded and hence uniformly integrable) and $a_t=m_t$, we see that, for almost every sequence $(m_t)$,
$$\frac1{n+1}\sum_{t=2}^n\sum_{i=1}^{m_t}\sum_{k=1}^\infty \frac{N_k(t,i-1)}{ (k+\delta_0)S_{t,i-1}(\delta_0)}
\xrightarrow[]{P} \mu \sum_k \frac{p_k^{(0)}}{(k+\delta_0) (\delta_0+2\mu)}$$
By a similar, but simpler, argument we see that
$$\frac1{n+1}\sum_{t=2}^n\sum_{i=1}^{m_t}\Bigl(\frac t{S_{t,i-1}(\delta_0)}\Bigr)^2\ra \frac{\mu}{(\delta_0+2\mu)^2}.$$
Since $p_k^{(0)}/(2\mu+\delta_0)=q_k^{(0)}/(k+\delta_0)$ by \eqref{EqDefqk},
the difference of the right sides of the last two  displays is $\nu_0$.
\end{proof}

The following is the main result of the paper.
\begin{theorem}
If $\delta_0$ is interior to the parameter set, then the MLE $\hat \delta_N$ satisfies,
for $\nu_0$ given in Lemma~\ref{prop-clt-mle-z},
    \begin{equation}    
        \label{eqn-mle-clt}
        \sqrt{n} (\hat \delta_n - \delta_0) \rightsquigarrow N(0,\nu_0^{-1}). 
    \end{equation}
    \label{thm-clt-mle}
\end{theorem}

\begin{proof}
By Theorem~\ref{prop-consistency} $\hat\delta_n$ tends to $\delta_0$, hence is
with probability tending to one interior to the parameter set, and must solve the likelihood equation
$ \iota'_n(\hat{\delta}_n)=0$.
By Taylor expansion there exists $\delta'_n$ between $\delta_0$ and $\hat \delta_n$ such that
    \[
0 = \iota'_n(\hat{\delta}_n)= \iota'_n(\delta_0) + \iota''_n(\delta'_n)(\hat{\delta}_n - \delta_0).
    \]
Using that $\iota'(\delta_0) = 0$, we can reformulate the preceding display as
\[
    \sqrt{n} (\hat{\delta}_n - \delta_0)  \iota_n''(\delta'_n) =- \sqrt{n}  \bigl(\iota'_n(\delta_0) -\iota'(\delta_0)\bigr).
    \]
The expression on the right is studied in Lemma~\ref{prop-clt-mle-z},
and seen to converge in distribution to $N(0,\nu_0) $. 

The second derivative takes the form
$$\iota''_n(\delta) =- \sum_{k=1}^\infty \frac{1}{(k+\delta)^2} \frac{N_{>k}(n) - R_{>k}(n)}{n+1} 
+ \frac{1}{n+1}\sum_{t=2}^n \sum_{i=1}^{m_t} \frac{t^2}{S_{t,i-1}^2(\delta)}. $$
By a similar argument as in the proof of Lemma~\ref{lemma-loglhd-uniform-conv} we see that
this converges in probability to the second derivative $\iota''(\delta)$, uniformly in $\delta$ in a neighbourhood
of $\delta_0$. Since $\delta'_n\ra\delta_0$ in probability and $\iota''$ is continuous,
it follows that  $\iota''_n(\delta'_n)\ra\iota''(\delta_0)$. The latter limit  is given by
$$\iota''(\delta_0) =- \sum_{k=1}^\infty \frac{p_{>k}^{(0)} - r_{>k}}{(k+\delta_0)^2}
+ \frac{\mu}{(2\mu+\delta_0)^2}=-\nu_0, $$
by \eqref{EqDefqk}. An application of Slutsky's lemma concludes the proof.
\end{proof}

It is shown in the preceding proof that the \emph{observed information} $-\iota''_n(\hat\delta_n)$ is 
a consistent estimator of the inverse asymptotic variance $\nu_0$. Thus, for
$\xi_\alpha$ quantiles from the normal distribution, 
$$\hat\delta_n\pm \xi_\alpha/\sqrt{-n\iota''_n(\hat\delta_n)}$$
is a confidence interval for $\delta$.

\section{Local Asymptotic Normality and Efficiency} 
The maximum likelihood estimator is well known to be asymptotically efficient, for instance in the sense of being
\emph{locally asymptotic minimax}, if the model is locally asymptotically normal and the distributional limit
is reached locally uniformly in the parameter. See e.g.\ Chapters~7 and~8 in \cite{vdvaart2000asymptotic}.
Local uniform convergence of $\hat\delta_n$ can be proved by a slight strengthening of 
Theorem~\ref{thm-clt-mle}: consider the behaviour under every sequence of parameters $\delta_0+h/\sqrt n$,
for $h\in \bbR$, instead of only under $\delta_0$. 

Local asymptotic normality of the model entails a quadratic  expansion of the local log likelihood ratio, given by,
for given $h\in\bbR$,
$$ \log \frac{l_n(\delta_0 + h/\sqrt{n}| D^{(n)})}{l_n(\delta_0| D^{(n)})}.$$
In view of \eqref{eqn-iota} this can be written in the form 
\begin{align*}
(n+1)\bigl[\iota_n(\delta_0 + h/\sqrt{n})-\iota_n(\delta_0)\bigr]
=\frac{n+1}{\sqrt n}h\iota_n'(\delta_0)+\frac12 \frac{n+1}{n}h^2\iota_n''(\delta'_n),
\end{align*}
for $\delta'_n$ between $\delta_0$ and $\delta_0+h/\sqrt n$.
By Lemma~\ref{prop-clt-mle-z} the first term on the right tends under $\delta_0$ to $h$ times a centered normal variable with variance $\nu_0$.
The second term on the right tends in probability to $-h^2\nu_0/2$, by the same arguments
as in the proof of Theorem~\ref{thm-clt-mle}.

This establishes the local asymptotic normality for our experiment and hence the efficiency of the MLE.

\section{The Case of Fixed Initial Degree}
\label{sec:fixed-initial-degree}
If the distribution of the initial degrees $m_t$ is degenerate at some natural number $m$, then
$m(n+1)$ is the total number of edges in the network at time $n$, and hence the 
number $m$ can be considered known given the snapshot of the network at time $n$.
In this case the MLE $\hat\delta_n$ is the root of the simpler version of \eqref{eqn-iota-def}, given by
\begin{equation}
\iota_n' (\delta) 
= \sum_{k = 1}^\infty \frac{p_{>k}(n)-1_{k<m}}{k+\delta} - \frac{1}{n+1} \sum_{t=2}^n \sum_{i=1}^m \frac{1}{\delta + 2m + (i-1)/t}.
\label{iota-prime-fixed}
\end{equation}
We have $q_k = (k+\delta)p_k/(2m + \delta) $, and the main theorem simplifies as follows.

\begin{proposition}
    \label{prop-clt-fixed} 
Suppose the initial degrees $m_t$ are fixed at some deterministic value $m \in \bbN$.  
If $\delta_0$ is interior to the parameter set, then the MLE $\hat \delta_n$ defined above satisfies,
\begin{equation}    
\label{eqn-mle-clt-fixed}
\sqrt{n} (\hat \delta_n - \delta_0) \rightsquigarrow N(0,\nu_0^{-1}),
    \end{equation}
where $\nu_0$ is defined by
\begin{equation*}
    \nu_0 = \sum_{k=1}^\infty \frac{m q_k^{(0)}}{ (k+\delta_0)^2} - \frac{ m}{ (2m + \delta_0)^2}.
\end{equation*}
\end{proposition}

Equation \eqref{iota-prime-fixed} relies only on the information
observable from the final snapshot of the network.  Hence in the case of fixed initial degrees,
as far as estimating $\delta$ is concerned, only
knowing the final snapshot is as informative as knowing  the historical evolution.

\section{Quasi-maximum-likelihood Estimator}
\label{sec:quasi-mle}

The solution $\hat\delta_n$ of the likelihood equation $\iota_n'(\delta)=0$ for $\iota_n'$ given by \eqref{eqn-iota-def}
depends on the sequence of initial degrees $m_1,\ldots, m_n$ through 
the quantities $R_{>k}(n) =2\cdot 1_{\{ m_1>k\}} + \sum_{t=2}^{n} 1_{\lrbkt{m_t >k}} $ and 
$S_{t,i-1}(\delta)/t=\delta+2\sum_{i=1}^{t-1}m_i/t+(i-1)/t$.
However in many applications, knowing the entire evolutional history is unrealistic, and the
sequence $m_1,\ldots, m_n$ may well be unobserved.  In this section we propose adaptations
of the estimator that depend only on the snapshot of the network at time $n$.

We first suppose that the distribution of the initial degrees is known; recall that
$r_{>k}=\bbP(m_t>k)$, $\mu=\bbE m_t$, and $\mu^{(2)}=\bbE m_t^2$, where the second moment is assumed finite.
Bearing in mind that \eqref{eqn-iota-def} is asymptotic to \eqref{EqDerivativeIota}, 
we replace $R_{>k}(n)$ in \eqref{eqn-iota-def} by $(n+1)r_{>k}$  and the second term on the r.h.s.\ of \eqref{eqn-iota-def} 
by  $1/(2+\delta/\mu)$. We then define a \emph{quasi-maximum-likelihood estimator} (QMLE)
$\tilde{\delta}_n$ as the root of the function
\begin{equation*}
    \tilde{\iota}_n'(\delta) = \sum_{k=1}^\infty \frac{1}{k + \delta} \left( \frac{ N_{>k}(n) }{ n+1} - r_{>k}\right) - \frac{1}{2 + \delta/\mu}. 
\end{equation*}
It is easy to see $\tilde{\iota}_n'$ is asymptotic to $\iota'$.  In fact
\begin{equation*}
    \tilde{\iota}_n' (\delta)- \iota'(\delta) = \sum_{k=1}^\infty \frac{p_{>k}(n) - p_{>k}^{(0)}}{ k + \delta},
\end{equation*}
which is the first term in  \eqref{eqn-separate-terms-iota-difference}. Hence $\tilde{\iota}'$ also converges uniformly
to $\iota'$ and consistency of $\tilde{\delta}_n$ follows by similar (but simpler) arguments as for $\hat\delta_n$. This gives
the following analogue of Theorem~\ref{prop-consistency}.

\begin{proposition}
    The QMLE $\tilde{\delta}_n$ is consistent, in probability under $\delta_0$ for every $\delta_0 \in (-a,b)$. 
\end{proposition}

We similarly can prove the asymptotic normality of $\tilde{\delta}_n$.
We first establish the asymptotics of $\tilde{\iota}_n'$ in the next proposition.

\begin{proposition}
    \label{prop-quasi-iota-prime}
    Under $\delta_0$, we have 
    \begin{equation}
    \sqrt{n} (\tilde{\iota}_n' - \iota')(\delta_0) \rightsquigarrow N(0, \tilde{\nu}_0 + \nu_0),
\end{equation}
where $\nu_0$ is given in Lemma~\ref{prop-clt-mle-z}, and $\tilde{\nu}_0$ is defined as the r.h.s.\ 
of \eqref{eqn-quasi-iota-var-limit}, and  depends only on the distribution of the initial degree and $\delta_0$.
\end{proposition}
For the proof of the above proposition, we need the two following lemmas.

\begin{lemma}
If $ X_1, X_2, \dots $ is a sequence of centered i.i.d.\ random variables with unit variance, 
and $g: \RR\to\RR$ is a measurable map with $\bbE g(X_1)=0$, $\bbE [g^2(X_1)]=1$ and $\bbE \bigl[X_1g(X_1)\bigr]=\rho$, 
then as $n \rightarrow \infty$, 
    \begin{gather}
        \label{eqn-aggregated-sample-mean-limit}
        \Bigl(\frac{1}{ \sqrt{n}} \sum_{i=1}^n \overline{X}_i,  \frac{1}{ \sqrt{n}} \sum_{i=1}^n g(X_i)\Bigr)^T \rightsquigarrow 
N_2\Bigl(\Bigl(\begin{matrix}0\\0\end{matrix}\Bigr),\Bigl(\begin{matrix}2 &\rho\\\rho&1\end{matrix}\Bigr)\Bigr), \\
        \frac{1}{\sqrt{n}} \sum_{i=1}^n \overline{X}_i^2 \xrightarrow{L^1} 0,
        \label{eqn-aggregated-sample-mean-second-limit}
    \end{gather}
    where $\overline{X}_i = (\sum_{j=1}^i X_j)/i$.  
    \label{lemma-aggregated-sample-mean-limit}
\end{lemma}
\begin{proof}
To prove \eqref{eqn-aggregated-sample-mean-limit} we first rewrite
    \begin{equation*}
        \frac{1}{\sqrt{n}} \sum_{i=1}^n \overline{X}_i  = \frac{1}{\sqrt{n}} \sum_{i=1}^n \frac{1}{i} \sum_{j=1}^i X_j = \frac{1}{\sqrt{n}} \sum_{j=1}^n X_j w_{j,n},
    \end{equation*}
    where $w_{j,n} := \sum_{i=j}^n 1/i$.  
We next apply the Lindeberg central limit theorem (see Theorem 2.27 of \cite{vdvaart2000asymptotic}) to the vectors
$(w_{j,n}X_j, g(X_j))$. Since \( \log ({n}/{j}) \le w_{j,n} \le \log (n/(j-1)) \) the variance of the first coordinate
and the covariance can be calculated as 
    \begin{align*}
        \Var\bigl(\frac{1}{\sqrt{n}}\sum_{i=1}^n \overline{X}_i \bigr) 
&= \frac{1}{n} \sum_{j=1}^n w_{j,n}^2 \rightarrow \int_{0}^1 (\log s)^2\, ds = \Gamma(3) = 2,\\
        \Covar\bigl(\frac{1}{\sqrt{n}}\sum_{i=1}^n \overline{X}_i,\frac{1}{\sqrt{n}}\sum_{i=1}^n g(X_i ) \bigr) 
&= \frac{1}{n} \sum_{j=1}^n w_{j,n}\rho \rightarrow -\rho\int_{0}^1 \log s\, ds = \rho\Gamma(2) = \rho.
    \end{align*}
The variance of the second coordinate is one by assumption.
To verify the Lindeberg condition we write
    \begin{align*}
        \frac{1}{n}\sum_{j=1}^n \bbE[X_j^2 w_{j,n}^2 1_{\{ |X_j| w_{j,n}/\sqrt{n} > \varepsilon \}}] & \lesssim \frac{1}{n} \sum_{j=1}^n w_{j,n}^2 \bbE[X_j^2 1_{\{ |X_j| > \sqrt{n} \varepsilon /\log n\}}]\\
        & \le\bbE[X_1^2 1_{\{ |X_1| > \sqrt{n} \varepsilon /\log n\}}] \frac{1}{n} \sum_{j=1}^n w_{j,n}^2.
    \end{align*}
The first inequality comes from that $w_{j,n} \le w_{1,n} \approx \log n$. As $X_1$ possesses a finite second moment, 
the right side tends to zero, as $n 
    \rightarrow \infty$.  Then by the Lindeberg central limit theorem, \eqref{eqn-aggregated-sample-mean-limit} holds. 

    Since $\bbE [\overline{X_i}^2]=\Var\overline{X}_i=1/i$, the expectation of the left side of 
\eqref{eqn-aggregated-sample-mean-second-limit} is equal to 
\begin{align*}
        \frac{1}{\sqrt{n}} \bbE\bigl[ \sum_{i=1}^ n \overline{X}_i^2\bigr] 
        & = \frac{1}{\sqrt{n}} \sum_{i=1}^n \frac{1}{i} \asymp \frac{\log n}{\sqrt{n}}.  
\end{align*}
This clearly tends to zero as $n\ra\infty$, proving 
\eqref{eqn-aggregated-sample-mean-second-limit}.
\end{proof}

\begin{lemma}
Suppose that $(X_n,Y_n,Z_n)$ are random vectors defined on a common probability space
such that $Y_n=g_n(Z_n)$ for measurable maps $g_n$, and, as $n \rightarrow \infty$,
    \begin{align*}
        X_n | Z_n&\rightsquigarrow N(0, \sigma^2), \text{\quad almost surely},\\
        Y_n &\rightsquigarrow N(0, \tau^2). 
    \end{align*}
    Then   $ X_n + Y_n \rightsquigarrow N(0, \sigma^2 + \tau^2)$.
 \label{lemma-sum-of-two-gaussian}
\end{lemma} 

\begin{proof}
For any two continuous, bounded functions $f$ and $g$, 
    \begin{align*}
        |\bbE[f(X_n) g (Y_n)] - \textstyle \int f d\Phi_\sigma \bbE [ g(Y_n)] | & = | \bbE[ \{ \bbE[ f(X_n) | Z_n] - \textstyle\int f d\Phi_\sigma \} g(Y_n)]| \\
        & \le \| g \|_\infty \bbE\bigl|\bbE[f(X_n)|Z_n] - \textstyle \int f d\Phi_\sigma \bigr| \rightarrow 0,
    \end{align*}
by the dominated convergence theorem, since $\bbE[f(X_n)|Z_n] \ra \int f\,d\Phi_\sigma$, almost surely,
by the  Portmanteau lemma.  Again by the Portmanteau lemma, we have
 \( \bbE[g(Y_n)] \rightarrow \textstyle \int g\, d\Phi_\tau\), and hence it becomes clear that
    \begin{equation*}
        \bbE[f(X_n) g(Y_n) ] \rightarrow \int f\, d \Phi_\sigma \int g\, d \Phi_\tau.
    \end{equation*}
This implies that the vectors $(X_n,Y_n)$ converge in distribution to a vector of two 
independent centered Gaussian variables with
variances $\sigma^2$ and $\tau^2$ (see Corollary~1.4.5 in \cite{vdvaartwellner}). Next the assertion of the lemma follows by
the continuous mapping theorem. 
\end{proof}

\begin{proof}[Proof of Proposition~\ref{prop-quasi-iota-prime}]
Decompose  $\tilde{\iota}_n' - \iota'$ as
    \begin{align*}
\sqrt n(\tilde{\iota}_n'- \iota')(\delta_0) = \sqrt n(\tilde{\iota}_n' - \iota_n')(\delta_0) + \sqrt n(\iota_n' - \iota')(\delta_0).
    \end{align*}
The second term on the r.h.s.\ has been studied in Lemma~\ref{prop-clt-mle-z},
and tends conditionally in distribution to a normal distribution with variance $\nu_0$ given $(m_t)_{t=1}^\infty$, almost surely,
while the first term depends only on $(m_t)_{t=1}^\infty$. In view of  
Lemma~\ref{lemma-sum-of-two-gaussian} it suffices to show that the first term tends
in distribution to a centered normal distribution with variance $\tilde \nu_0$.

We have
\begin{align*}
(\tilde{\iota}_n' - \iota_n')(\delta_0)  
&= \sum_{k=1}^\infty \frac{ r_{>k}(n) - r_{>k}}{k + \delta_0} + 
\frac{1}{n+1} \sum_{t=2}^n \sum_{i=1}^{m_t} \frac{t}{ S_{t,i-1}(\delta_0)} - \frac{1}{ 2 + \delta_0/\mu}\\
& = \frac1{n+1}\sum_{t=2}^n \sum_{ k =1}^\infty \frac{ 1_{\{ m_t > k\}} - r_{>k}}{ k + \delta_0}
  + \frac1{n+1} \sum_{t=2}^n \frac{m_t-\mu}{ \delta_0 + 2 \mu}\\
&\qquad+\frac{2\mu}{n+1} \sum_{t=2}^n \frac{\mu-\overline{m}_{t-1}}{( \delta_0 + 2 \mu)^2 } +A_n+B_n+C_n,
    \end{align*}
where
\begin{align*}
A_n&=\frac2{n+1}\sum_{ k =1}^\infty \frac{ 1_{\{ m_1 > k\}} - r_{>k}}{ k + \delta_0}
+\Bigl(\frac{n - 1}{n+1}-1\Bigr)\frac{1}{2+\delta_0/\mu},\\
B_n&=\frac1{ n + 1} \sum_{t=2}^n  \sum_{i=1}^{m_t} \left( \frac{t}{ S_{t,i-1}(\delta_0)} - \frac{t}{ S_{t,0}(\delta_0)}\right) 
=\frac1{n+1} \sum_{t=2}^n \sum_{i=1}^{m_t} \frac{t (i-1)}{ S_{t,i-1}(\delta_0) S_{t,0}(\delta_0)}  \\
C_n&=\frac2{n+1} \sum_{t=2}^n \frac{(m_t-\mu)(\mu-\overline{m}_{t-1})}{(\delta_0 + 2 \overline{m}_{t-1}) ( \delta_0 + 2 \mu) }
        + \frac{4\mu}{n+1} \sum_{t=2}^n \frac{(\overline{m}_{t-1} - \mu)^2 }{ (\delta_0 + 2 \mu)^2 (\delta_0 + 2 \overline{m}_{t-1}) }.
    \end{align*}
    Clearly $\sqrt n A_n\ra 0$ in probability. Furthermore, since $m_t$ is independent from $\overline{m}_{t-1}$ and $m_t \ge 1$, $B_n$ is nonnegative with 
$$\bbE [B_n] \le \frac1{n+1} \bbE\Bigl[ \sum_{t=2}^n \frac{t m_t^2}{ S^2_{t,0}(\delta_0)}\Bigr] 
\le \frac{\mu^{(2)}}{(\delta_0 + 2)^2( n+1) } \sum_{t=2}^n \frac{1}{t }  \lesssim \frac {\log n}{n}.$$
Hence $\sqrt n B_n\ra 0$ in probability. The second term of $\sqrt n C_n$ tends to zero in
mean by \eqref{eqn-aggregated-sample-mean-second-limit}.  Finally, using that $m_t-\mu$ has mean zero and is independent 
of $\overline{m}_{t-1}$, we see that the terms of the first sum of $C_n$ are uncorrelated, and hence the second moment of the first term of $C_n$ is bounded above by
$$\frac{4\mu^{(2)}}{(n+1)^2 (\delta_0+2\mu)^2 (\delta_0 + 2)^2} \sum_{t=2}^n  \bbE[ (\mu-\overline{m}_{t-1})^2 ] \lesssim 
\frac{\log n}{n^2}.$$
Hence $\sqrt n C_n\ra0$ in probability as well. 

It follows that  $\sqrt n(\tilde{\iota}_n' - \iota_n')(\delta_0)$ has the same limit distribution as the sequence
$n^{-1/2}\sum_{t=2}^n \bigl[g(m_t)-2\mu(\overline{m}_{t-1}-\mu)/(\delta_0+2\mu)^2)\bigr]$, for $g$ defined by
$$g(m)=\sum_{ k =1}^\infty \frac{ 1_{\{ m > k\}} - r_{>k}}{ k + \delta_0}+
\frac{m -\mu}{ \delta_0 + 2 \mu}.$$
An application of Lemma~\ref{lemma-aggregated-sample-mean-limit} shows that this sequence
is asymptotically normal with mean zero and variance
\begin{equation}
\label{eqn-quasi-iota-var-limit}
\tilde\nu_0:=\Var[g(m_t)]+ \frac{8\mu^2\Var[m_t]}{(\delta_0+2\mu)^4}-\frac{4\mu}{(\delta_0+2\mu)^2}\bbE[g(m_t)m_t].
\end{equation}
The proof of the proposition is complete.
\end{proof}

It becomes immediate that a parallel of Theorem~\ref{thm-clt-mle} holds for the QMLE.  
\begin{theorem}
    If $\delta_0$ is interior to the parameter set, then the QMLE $\tilde{\delta}_n$ satisfies
    \begin{equation}
        \sqrt{n} (\tilde{\delta}_n - \delta_0 ) \rightsquigarrow N(0,(\nu_0 + \tilde{\nu}_0)/\nu_0^2).
    \label{eqn-qmle-clt}
    \end{equation}
    \label{thm-qmle-clt}
\end{theorem}
\begin{proof}
    The second order derivative of $\tilde{\iota}_n$ takes the form 
    \begin{equation*}
\tilde{\iota}_n''(\delta) 
= -\sum_{k=1}^\infty \frac{1}{(k+\delta)^2 } \left( \frac{N_{>k}(n)}{n+1} - r_{>k}\right) + \frac{\mu}{( 2\mu+\delta)^2}.
    \end{equation*}
This is asymptotic to $\iota''(\delta)$, just as $\iota''_n(\delta)$, uniformly in $\delta$ in a neighbourhood of
$\delta_0$, and the value at $\delta_0$ tends to $-\nu_0$, as $n \rightarrow \infty$. 
Therefore the result follows by the same argument as in the proof of Theorem~\ref{thm-clt-mle}. 
\end{proof}

The QMLE is only available if the distribution of the initial degrees is known. The mean initial degree
$\mu$ can always be estimated from the snapshot at time $n$ by the 
total number of edges divided by the number of vertices, i.e.\ by $\hat\mu_n=\sum_k kN_k(n)/(2n)$.
The QMLE could be adapted by replacing $\overline{m}_{t-1}$ in the terms $S_{t,i-1}(\delta)$ by $\hat \mu_n$.
If the initial degree distribution would be known to concentrate on two known values,
its mean would fix the full distribution (the two probabilities), and the QMLE could be
made completely data-dependent. One strategy would be to solve 
the quantities $R_{>k}(n)$ using the equations $\sum_{t=1}^n1_{m_t=a}+\sum_{t=1}^n 1_{m_t=b}=n$ and
$a\sum_{t=1}^n1_{m_t=a}+b\sum_{t=1}^n 1_{m_t=b}=\sum_k kN_k(n)/2$, if $a$ and $b$ are the two possible values
of $m_t$. However, in general the initial degree distribution appears to be confounded with the
preferential attachment model and lack of knowledge of the distribution may hamper estimation of $\delta$.


\section{Simulation Study}
\label{sec-numerical}
In this section we demonstrate the power of the MLE by applying it in the setting of the
famous paper \cite{Barabasi99emergenceScaling}.  We show that the MLE provides a better estimate
for the power-law exponent  than the ad-hoc estimator in this paper, and conclude that it is more informative about the limiting degree
distribution than the empirical degree distribution.  
We also offer an explanation for this somewhat counter-intuitive  phenomenon.

\subsection{\textit{On the shoulder of the giants}}
We present our simulation results to pay tribute to \cite{Barabasi99emergenceScaling}, who
considered (only)  the so-called linear preferential attachment model, i.e.\ $\delta_0 =0 $ (or
$f(k) = k$). In this case the limiting degree distribution can be explicitly given as $p_k = 6/(k(k+1)(k+2))$.  
Figure~$2$ of \cite{Barabasi99emergenceScaling} presented a simulation of $150,000$ vertices
with fixed number of edges $m = 5$ and true parameter $\delta_0 = 0$.  Fitting a straight line
to the log empirical degree against the log degree gave a slope $2.9$, which suggests
a parameter $\delta$ equal to $2.9-3=-0.1$, different from the true
value $\delta_0 = 0$.  Below we show that the \textsc{MLE} gives an estimate of $\delta$ with a much smaller error,
and show that the sample mean and variance of the MLE are close to the asymptotic values obtained in Theorem~\ref{thm-clt-mle}.  
Since the most important characteristic of the limiting degree distribution is its power-law exponent $\tau$, 
which depends on $\delta$ through the relation $\tau = 3 + \delta/m$ if we assume the fast decay of $r_k$ and hence no influence on the power-law exponent from the initial degree distribution, the MLE in turn also gives a better 
insight in the limiting degree distribution. 

As mentioned in Section~\ref{sec:intro} a practical problem of fitting a straight line to the plot of the log empirical degree
$\log p_k(n)$ against $\log k$, is that the variance of the former blows up with increasing degree $k$.
This makes a careful choice of the cutoff mandatory, but difficult, as is illustrated in Figure~\ref{figure:loglog}, which
is an independent replicate of Figure~$2$(A) of  \cite{Barabasi99emergenceScaling}. Actually, in our simulations we found it to be non-trivial to 
obtain the estimate $2.9$ of the negative of the slope (as found in
\cite{Barabasi99emergenceScaling}).  In contrast, the MLE implicitly weighs the information in the various degrees,
and automatically yields  accurate estimates.

We studied the sampling distribution of the MLE by generating $3,500$ independent replicates of the linear
PA graph, and computing $\hat \delta$ for each replicate. A histogram of the 3,500 
maximum likelihood estimates $\hat{\delta}$ is plotted in Figure~\ref{figure:mle-plot},
and numerical summaries are given in Table~\ref{table:mle-summary}.
We compared these to the asymptotic distribution given  in Theorem~\ref{thm-clt-mle}, and
a ``best fitting'' normal distribution with mean and variance taken equal to the sample mean and variance
of the 3,500 replicates. Inspection of Figure~\ref{figure:mle-plot} and Table~\ref{table:mle-summary} gives
the following observations.
\begin{itemize}
\item The MLE is very accurate.  In 3,500 replicates the \textit{worst} estimate deviated about $0.06$ from the truth $\delta_0 = 0$
(Table~\ref{table:mle-summary}), better than the error $2.9-3=-0.1$ obtained in \cite{Barabasi99emergenceScaling}. 
\item The bias of the MLE is estimated at about  $0.00082$ and the variance at $0.00033$ (Table~\ref{table:mle-summary}, 
``Mean'' and ``Samp.\ Var''), showing that the bias gives
a negligible contribution to the mean square error, as predicted by Theorem~\ref{thm-clt-mle}. 
The estimated variance is  remarkably close to the asymptotic variance predicted
 by Theorem~\ref{thm-clt-mle} (given in the column  ``Pred.\ Var.'' of Table~\ref{table:mle-summary}). 
\item From the relation  $\tau = 3+ \delta/m$ we obtain the sample bias and variance in the estimates of 
the power-law exponent.
\item The sampling distribution of $\hat{\delta}_n$ with $n = 150,000$ and fixed $m = 5$ is quite close to a normal distribution
(Figure~\ref{figure:mle-plot}), although it appears slightly skewed to the right.
\item The normal law that fits best to the sampling distribution is very close to the 
asymptotic normal distribution (red  and blue curves in Figure~\ref{figure:mle-plot}).
\end{itemize}

\begin{figure}[htbp]
 \centering
 \includegraphics[height=.4\textheight]{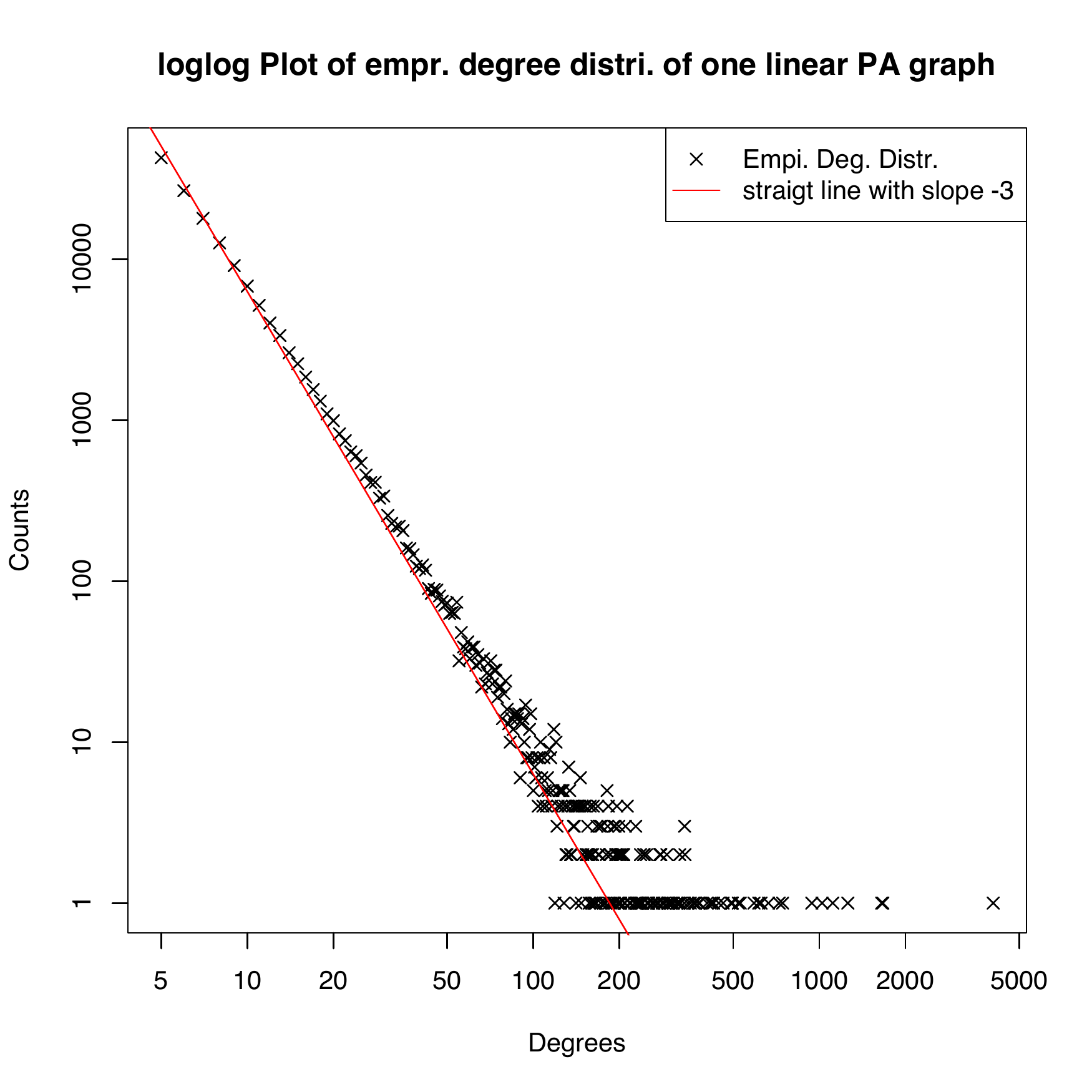}
\caption{Plot of log Empirical Degree $p_k(n)$ (vertical axis) versus log degree $\log k$ of in one simulation of a Linear Preferential Attachment
graph with $m=5$ and $n=150,000$.}
 \label{figure:loglog}
\end{figure}

\begin{figure}[htbp]
 \centering
 \includegraphics[height=0.4\textheight]{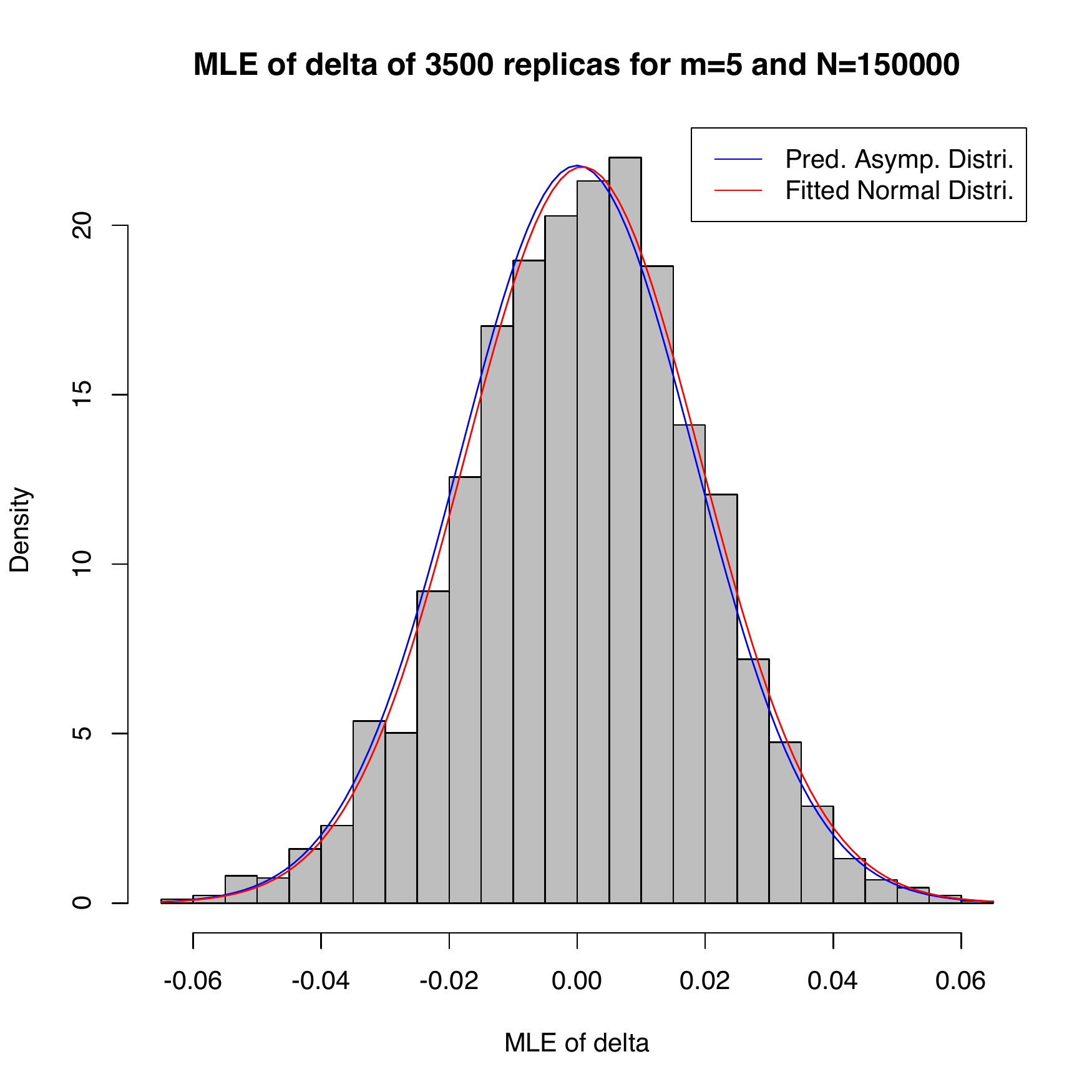}
 \caption{Histogram of 3,500 replicates of the MLE in the Linear Preferential Attachment model with $m=5$ and $n=150,000$.}
 \label{figure:mle-plot}
\end{figure}

\begin{table}[htbp]
    \centering
    \begin{tabular}[http]{ c c c c c c c c }
       Min.\ &  Median &  Mean  &  Max.\ & Samp.\ Var.\ & Pred.\ Var.\ \\
       -0.064580  & 0.0013350 &  0.0008268 & 0.064090 & 0.00033715 & 0.00033594 
    \end{tabular}
\caption{Summary of 3,500 replicates of the MLE in the linear Preferential Attachment model with  $m=5$ and $n = 150,000$.
``Pred.\ Var.'' gives the variance as predicted by the asymptotic distribution of the MLE.}
    \label{table:mle-summary}
\end{table}

\subsection{\textit{The majority rules}}
As we pointed out in the previous section, the MLE offers more information about the limiting degree
distribution than the empirical degree distribution.  We offer a possible explanation why this
happens.  If we only utilize the empirical degree distribution and try to fit a power-law, then 
\textit{essentially} we must restrict ourselves to  nodes with high degrees, and neglect the absolute majority with low
degrees.  However if the universal mechanism of preferential attachment is 
responsible for the evolution of the whole network, then this mechanism is also responsible for the nodes
with low degrees, whence nodes with low degrees also provide information about the underlying
mechanism.  The MLE takes account of all the nodes, and also weighs their relevance in an automated
manner,  and uses more information than present in just nodes of high degrees.  In this sense, the majority (the nodes with low degrees) wins
over the minority (the nodes with high degrees) and thus ``rules''.

\section{Acknowledgements}
The paper is in part inspired by the first author's master's thesis project (\cite{gao2011imdb}), where the modeling of the movie-actor network with the preferential attachment models was studied.  
Though the first author has moved to Shanghai, the work of this paper was almost exclusively conducted during
his PhD research in Leiden.  The authors thank the anonymous referees to their useful comments, which helped significanty in improving the paper. 

\section*{References}
\bibliography{pam}

\end{document}